\documentclass[a4paper,10pt]{amsart}

\usepackage{amsmath,amsfonts,amssymb}
\usepackage[utf8x]{inputenc}
\usepackage[english]{babel}
\usepackage[T1]{fontenc}
\usepackage{lmodern}
\usepackage{graphicx}
\usepackage[all]{xy}
\usepackage{latexsym}
\usepackage{fancyhdr}
\usepackage{lscape}
\usepackage{url}
\usepackage{ifthen}
\usepackage{booktabs}
\usepackage{paralist}
\usepackage{array}
\usepackage{multirow}
\usepackage{epstopdf}
\usepackage{ifpdf}
\usepackage{caption}
\usepackage{subcaption}
\usepackage{multicol}
\usepackage{textcomp}

\usepackage[hypertexnames=false,
backref=page,
    pdftex,
    pdfpagemode=UseNone,
    breaklinks=true,
    extension=pdf,
    colorlinks=true,
    linkcolor=blue,
    citecolor=blue,
    urlcolor=blue,
]{hyperref}

\allowdisplaybreaks[1]

\newcommand{\referenza}{}

\newtheorem{prop}{Proposition}[section]
\newtheorem{thm}[prop]{Theorem}
\newtheorem*{thm*}{Theorem \referenza}

\newtheorem{lemma}[prop]{Lemma}

\theoremstyle{definition}
\newtheorem{defi}[prop]{Definition}
\newtheorem{rem}[prop]{Remark}

\newcommand{\Z}{\mathbb{Z}}
\newcommand{\R}{\mathbb{R}}
\newcommand{\C}{\mathbb{C}}
\newcommand{\st}{\;:\;}

\newcommand{\sspace}{\text{\--}}
\newcommand{\ssspace}{\text{\textdblhyphen}}

\newcommand{\scalar}[2]{\left\langle #1 \,\middle|\, #2 \right\rangle}

\DeclareMathOperator{\im}{i}
\DeclareMathOperator{\imm}{im}
\DeclareMathOperator{\de}{d}
\DeclareMathOperator{\id}{id}
\DeclareMathOperator{\End}{End}

\DeclareMathOperator{\Tot}{Tot}

\DeclareMathOperator{\Cliff}{Cliff}

\newcommand{\del}{\partial}
\newcommand{\delbar}{\overline{\del}}

\title{On cohomological decomposition of generalized-complex structures}

\author{Daniele Angella}

\address[Daniele Angella]{Istituto Nazionale di Alta Matematica}
\curraddr{Dipartimento di Matematica e Informatica\\
Universit\`{a} degli Studi di Parma\\
Parco Area delle Scienze 53/A, 43124\\
Parma, Italy}

\email{daniele.angella@math.unipr.it}

\author{Simone Calamai}

\address[Simone Calamai]{Dipartimento di Matematica e Informatica ``Ulisse Dini''\\
Università di Firenze\\
via Morgagni 67/A, 50134\\
Firenze, Italy}

\email{scalamai@math.unifi.it}

\author{Adela Latorre}

\address[Adela Latorre]{
Departamento de Matem\'aticas -- IUMA\\
Universidad de Zaragoza\\
Campus Plaza San Francisco, 50009\\
Zaragoza, Spain}

\email{adela@unizar.es}

\keywords{almost-complex, generalized-complex, cohomology, Iwasawa manifold, $\mathcal{C}^\infty$-pure-and-full, Brylinski conjecture.}
\thanks{The first author is granted with a research fellowship by Istituto Nazionale di Alta Matematica INdAM and is supported by the Project PRIN ``Varietà reali e complesse: geometria, topologia e analisi armonica'', by the Project FIRB ``Geometria Differenziale e Teoria Geometrica delle Funzioni'', and by GNSAGA of INdAM. The second author is supported by GNSAGA of INdAM and by the Project PRIN ``Varietà reali e complesse: geometria, topologia e analisi armonica''. The third author is supported by Project MICINN (Spain) MTM2011-28326-C02-01 and by a DGA predoctoral scholarship.}
\subjclass[2010]{32Q60, 53D18, 57T15}

\begin{document}

\begin{abstract}
 We study properties concerning decomposition in cohomology by means of generalized-complex structures. This notion includes the $\mathcal{C}^\infty$-pure-and-fullness introduced by Li and Zhang in the complex case and the Hard Lefschetz Condition in the symplectic case. Explicit examples on the moduli space of the Iwasawa manifold are investigated.
\end{abstract}

\maketitle

\section*{Introduction}

The decomposition of harmonic forms on compact K\"ahler manifolds into bi-graded components is a strong result in Hodge theory. Therefore, one would also like to extend results on cohomological decompositions to weaker structures, maybe relying on their complex, symplectic, or Riemannian aspects. We recall, for example, the theory initiated by J.-L. Brylinski on Hodge theory for symplectic manifolds \cite{brylinski}. The analogies and differences between the results obtained in the complex and symplectic cases acquire a deeper meaning when they are framed into the generalized-complex setting. Tools from generalized-complex geometry have been recently used to develop a Hodge theory for SKT structures in \cite{cavalcanti-skt}. The aim of this note is to provide a notion of cohomological-decomposability on generalized-complex manifolds, showing its coherence with already-known notions for complex and symplectic structures. We consider that having a comparison frame between these two parallel cases may inspire further results, in either complex or symplectic geometry.

\medskip

Let $(X,\, J)$ be a compact complex manifold. There is a natural bi-graded subgroup of the de Rham cohomology of $X$, given by the image of the map
$$ H^{\bullet,\bullet}_{BC}(X) \;:=\; \frac{\ker\del\cap\ker\delbar}{\imm\del\delbar} \longrightarrow H^\bullet_{dR}(X;\C) \;. $$
Observe that the surjectivity of this map would yield to a cohomological decomposition of de Rham cohomology related to the complex structure.
In fact, the stronger property of the above map being an isomorphism is called the $\partial\overline\partial$-Lemma property.
Furthermore, it should be noted that these subgroups also make sense in a more general framework: that of almost-complex structures $J$ on $X$. Indeed, it suffices to set
$$ H^{(p,q)}_{J}(X) \;:=\; \left\{ \left[\alpha\right] \in H^{p+q}_{dR}(X;\C) \st \alpha\in\wedge^{p,q}_JX \right\} \;\subseteq\; H^{p+q}_{dR}(X;\C) \;. $$
In the integrable case, $H^{(p,q)}_{J}(X) = \imm\left(H^{p,q}_{BC}(X)\to H^{p+q}_{dR}(X;\C) \right)$.
Such subgroups have been studied by T.-J. Li and W. Zhang in \cite{li-zhang} when investigating symplectic cones on almost-complex manifolds. It is important to mention that, in general, these subgroups may not yield to a decomposition of the de Rham cohomology. In this sense, the result by T. Dr\v{a}ghici, T.-J. Li, and W. Zhang in \cite[Theorem 2.3]{draghici-li-zhang} appears as a very specific property of compact $4$-dimensional manifolds: it states that any almost-complex structure $J$ on a compact $4$-dimensional manifold $X^4$ satisfies
\begin{eqnarray*}
 H^2_{dR}(X^4;\R) &=& \left\{ \left[\alpha\right]\in H^2_{dR}(X^4;\R) \st J\alpha=\alpha \right\} \\[5pt]
 &\oplus & \left\{ \left[\alpha\right]\in H^2_{dR}(X^4;\R) \st J\alpha=-\alpha \right\} \;.
\end{eqnarray*}
Another interesting example is the Iwasawa manifold $\mathbb{I}_3$, which is one of the simplest non-K\"ahler example of complex threefold \cite{fernandez-gray}. With respect to its natural holomorphically-parallelizable complex structure, the map $H^{\bullet,\bullet}_{BC}(\mathbb{I}_3) \to H^\bullet_{dR}(\mathbb{I}_3;\C)$ is surjective (see \cite[Theorem 3.1]{angella-tomassini-1} and \S\ref{sec:iwasawa}). The same holds when one endows the underlying differentiable manifold of $\mathbb{I}_3$ with the Abelian complex structure given in \S\ref{sec:iwasawa} (see \cite{latorre-ugarte}). Such examples show that this kind of decomposition is a strictly weaker property than the $\partial\overline\partial$-Lemma.

\medskip

Consider now a compact symplectic manifold $(X,\, \omega)$. A symplectic Hodge theory was proposed by J.-L. Brylinski in \cite{brylinski} and further results in this direction were obtained, among others, by O. Mathieu \cite{mathieu}, D. Yan \cite{yan}, V. Guillemin \cite{guillemin}, and G.~R. Cavalcanti \cite{cavalcanti-phd}. Recently, L.-S. Tseng and S.-T. Yau introduced and studied some symplectic cohomologies \cite{tseng-yau-1, tseng-yau-2, tseng-yau-3, tsai-tseng-yau}. The group
$$ SH^\bullet_{BC}(X) \;:=\; \frac{\ker\de\cap\ker\de^\Lambda}{\imm\de\de^\Lambda} \;, $$
also denoted by $H^\bullet_{\de+\de^\Lambda}$ in \cite{tseng-yau-1}, plays the same role as the Bott-Chern cohomology for complex manifolds (here, $\de^\Lambda:=\left[\de,-\iota_{\omega^{-1}}\right]$). As shown in Proposition \ref{prop:Cpf-sympl}, the surjectivity of the natural map
$$ SH^\bullet_{BC}(X) \longrightarrow H^\bullet_{dR}(X;\R) $$
induced by the identity turns out to be equivalent to the property that every de Rham cohomology class admits a $\de$-closed, $\de^\Lambda$-closed representative; that is, the Brylinski conjecture \cite[Conjecture 2.2.7]{brylinski} holds. This is also equivalent to the Hard Lefschetz Condition \cite[Corollary 2]{mathieu}, \cite[Theorem 0.1]{yan} and to the so-called $\de\de^\Lambda$-Lemma \cite[Proposition 1.4]{merkulov}, \cite{guillemin}, \cite[Theorem 5.4]{cavalcanti-phd}. (In fact, note that the spectral sequences associated to the bi-differential complex $\left( \wedge^\bullet X,\, \de,\, \de^\Lambda \right)$ degenerate at the first level \cite[Theorem 2.3.1]{brylinski}, \cite[Theorem 2.5]{fernandez-ibanez-deleon}, in contrast to the complex case.)

\medskip

Generalized-complex geometry was introduced by N. Hitchin \cite{hitchin} and studied, among others, by his students M. Gualtieri \cite{gualtieri-phd, gualtieri-annals} and G.~R. Cavalcanti \cite{cavalcanti-jgp} (see also \cite{cavalcanti-impa}). It provides a unified framework for both symplectic and complex structures. In fact, any generalized-complex structure is locally equivalent to the product of the standard complex structure on $\C^k$ and the standard symplectic structure on $\R^{2n-2k}$; see \cite[Theorem 4.35]{gualtieri-phd}, \cite[Theorem 3.6]{gualtieri-annals}.

\medskip

In this note, we study some results concerning decomposition in cohomology induced by generalized-complex structures $\mathcal{J}$ on a compact manifold $X$. More precisely, we consider the property that the natural map
$$ GH^\bullet_{BC}(X) \;:=\; \frac{\ker\del\cap\ker\delbar}{\imm\del\delbar} \longrightarrow GH_{dR}(X) $$
is surjective (here, $\del$ and $\delbar$ are the components of the exterior differential with respect to the graduation induced by $\mathcal{J}$ on the space of complex forms). In the special cases when $\mathcal{J}$ is induced by either a symplectic or a complex structure, we compare this property with the already-known notions (see Proposition \ref{prop:Cpf-sympl} and Proposition \ref{prop:Cpf-cplx}).

As an explicit example, we study structures on the real nilmanifold underlying the Iwasawa manifold. In particular, we focus on the holomorphically-parallelizable and the Abelian complex structures mentioned above, which provide a cohomological decomposition in complex sense (see \S\ref{sec:iwasawa}). In fact, they induce a cohomological decomposition in generalized-complex sense, too. As observed in \cite{ketsetzis-salamon}, they belong to two different components of the moduli space of left-invariant complex structures on the differentiable Iwasawa manifold. Nevertheless, G.~R. Cavalcanti and M. Gualtieri showed in \cite{cavalcanti-gualtieri} that such structures can be connected by a path of generalized-complex structures, which are given as $\beta$-transform and $B$-transform of a generalized-complex structure $\rho$. We study the cohomological decomposition property of such $\rho$, proving that the natural map from the generalized-Bott-Chern to the generalized-de Rham cohomology is surjective (see \S\ref{subsec:iwasawa-gencplx-path}). Furthermore, we provide another curve of generalized-complex structures connecting these two complex structures, but arising as $\beta$-transform and $B$-transform of a curve $\left\{J_t\right\}_{t\in[0,1]}$ of almost-complex structures. However, we show that these $J_t$s do not satisfy cohomological decomposition in the sense of Li and Zhang.

\bigskip

{\itshape Acknowledgments.}
The authors are greatly indebted to Xiuxiong Chen, Adriano Tomassini, and Luis Ugarte for their constant support and encouragement.
This work was in part initially conceived during the stay of the first author at Universidad de Zaragoza thanks to a grant by INdAM: the first author would like to thank the Departamento de Matem\'aticas for the warm hospitality.
Thanks are also due to Magda Rinaldi for useful discussions.

\section{Preliminaries and notation}
In this section, we recall the main definitions and results in generalized-complex geometry, in order to fix the notation. See, e.g., \cite{cavalcanti-impa} and the references therein for more details.

\subsection{Generalized-complex structures}
Let $X$ be a compact differentiable manifold of dimension $2n$. Consider the bundle $TX \oplus T^*X$ endowed with a natural symmetric pairing given by
\begin{eqnarray*}
 && \scalar{X+\xi}{Y+\eta}\;:=\; \frac{1}{2}\,\left(\xi(Y)+\eta(X)\right) \;.
\end{eqnarray*}

A {\em generalized-almost-complex structure} on $X$, \cite[Definition 4.14]{gualtieri-phd}, is a $\scalar{\sspace}{\ssspace}$-orthogonal endomorphism $\mathcal{J}\in \End(TX\oplus T^*X)$ such that $\mathcal{J}^2=-\id_{TX\oplus T^*X}$.

\medskip

Following \cite[\S3.2]{gualtieri-phd}, \cite[\S2]{gualtieri-annals}, consider the \emph{Courant bracket} on $TX\oplus T^*X$, 
\begin{eqnarray*}
&& \left[X+\xi,\, Y+\eta\right] \;:=\; \left[X,\, Y\right] + \mathcal{L}_X\eta - \mathcal{L}_Y\xi - \frac{1}{2}\, \de \left(\iota_X\eta-\iota_Y\xi\right) \;,
\end{eqnarray*}
and its associated Nijenhuis tensor for $\mathcal{J}\in\End\left(TX\oplus T^*X\right)$,
\begin{eqnarray*}
  \mathrm{Nij}_{\mathcal{J}} &:=& -\left[ \mathcal{J}\,\sspace ,\, \mathcal{J}\,\ssspace\right] + \mathcal{J} \left[ \mathcal{J}\,\sspace ,\, \ssspace \right] + \mathcal{J} \left[ \sspace ,\, \mathcal{J}\,\ssspace \right] + \mathcal{J} \left[ \sspace ,\, \ssspace \right] \;.
\end{eqnarray*}
(As a matter of notation, $\iota_{X}\in \End^{-1}\left(\wedge^\bullet X\right)$ denotes the interior product with $X\in \mathcal{C}^\infty(X;TX)$, and $\mathcal{L}_X:=\left[\iota_X,\, \de\right]\in \End^0\left(\wedge^\bullet X\right)$ denotes the Lie derivative along $X\in \mathcal{C}^\infty(X;TX)$.)

A {\em generalized-complex structure} on $X$ is a generalized-almost-complex structure $\mathcal{J}\in \End(TX\oplus T^*X)$ such that $\mathrm{Nij}_{\mathcal{J}}=0$, \cite[Definition 4.14, Definition 4.18]{gualtieri-phd}, \cite[Definition 3.1]{gualtieri-annals}.

\subsection{Graduation on forms}

Generalized-complex structures provide a graduation on the space of complex differential forms \cite[\S4.4]{gualtieri-phd}, \cite[Proposition 3.8]{gualtieri-annals}.

\medskip

Consider a $2n$-dimensional differentiable manifold $X$ endowed with a generalized-almost-complex structure $\mathcal{J}$.

Let $L$ be the $\im$-eigenspace of the $\C$-linear extension of $\mathcal{J}$ to $\left(TX \oplus T^*X\right)\otimes_\R \C$. Consider the complex rank $1$ sub-bundle $U$ of $\wedge^\bullet X \otimes_\R \C$ generated by a complex form $\rho$ whose Clifford annihilator is precisely $L=\left\{v\in \left(TX \oplus T^*X\right) \otimes_\R \C \st v\cdot \rho=0\right\}$. Here, the operation denotes the Clifford action of $TX\oplus T^*X$ on the space of differential forms on $X$ with respect to $\scalar{\sspace}{\ssspace}$, i.e.,
$$\begin{array}{rcl}
\Cliff\left(TX\oplus T^*X\right) \times \wedge^\bullet X  &\longrightarrow & \wedge^{\bullet-1} X \oplus \wedge^{\bullet+1} X\\[2pt]
\left(  (X+\xi), \, \varphi  \right) &\longmapsto & (X+\xi) \cdot \varphi \;:=\; \iota_X\varphi + \xi\wedge\varphi
\end{array}$$
as well as its bi-$\C$-linear extension.

For each $k\in\Z$, define
$$ U^k \;:=\; \wedge^{n-k}\bar L \cdot U \;\subseteq\; \wedge X \otimes_\R \C \;. $$

\medskip

By \cite[Theorem 4.3]{gualtieri-phd}, \cite[Theorem 3.14]{gualtieri-annals}, the condition $\mathrm{Nij}_{\mathcal{J}}=0$ is equivalent to the property
$$ \de U^\bullet \subset U^{\bullet+1} \oplus U^{\bullet-1} \;. $$
Therefore, one has \cite[\S4.4]{gualtieri-phd}, \cite[\S3]{gualtieri-annals}
$$ \de \;=\; \del + \delbar \;, $$
where
\begin{eqnarray*}
\del\lfloor_{U^\bullet} &:=& \pi_{U^{\bullet+1}}\circ \de\lfloor_{U^\bullet} \colon \colon U^\bullet \to U^{\bullet+1} \;, \\[5pt]
\delbar\lfloor_{U^\bullet} &:=& \pi_{U^{\bullet-1}}\circ \de\lfloor_{U^\bullet} \colon \colon U^\bullet \to U^{\bullet-1} \;.
\end{eqnarray*}

\smallskip

Now, let us recall the notion of $B$-field transform \cite[\S3.3]{gualtieri-phd} and see how it may affect the initial graduation of forms for a given generalized-complex structure $\mathcal{J}$ defined on $X$.

Consider a $\de$-closed $2$-form $B$, viewed as a map $TX\to T^*X$. Consider the generalized-complex structure given by
$$ \mathcal{J}^B \;:=\; \exp \left(-B\right) \, \mathcal{J} \, \exp B, \quad \text{ where } \quad \exp B \;=\;
\left(
\begin{array}{c|c}
 \id_{TX} & 0 \\
\hline
 B & \id_{T^*X}
\end{array}
\right) \;.$$
Then, the $\Z$-graduation is given by \cite[\S2.3]{cavalcanti-jgp}
$$ U^{\bullet}_{\mathcal{J}^B} \;=\; \exp B \wedge U^\bullet_{\mathcal{J}} $$
and in particular, \cite[\S2.3]{cavalcanti-jgp},
$$ \del_{\mathcal{J}^B} \;=\; \exp \left(-B\right) \, \del_{\mathcal{J}} \, \exp B \qquad \text{ and } \qquad \delbar_{\mathcal{J}^B} \;=\; \exp \left(-B\right) \, \delbar_{\mathcal{J}} \, \exp B \;. $$

\subsection{Complex and symplectic structures}

Complex and symplectic structures can be seen as very special examples of generalized-complex structures.

\medskip

Consider a $2n$-dimensional differentiable manifold $X$ endowed with a generalized-complex structure $\mathcal{J}$. We recall that, \cite[\S4.3]{gualtieri-phd}, \cite[Definition 3.5]{gualtieri-annals}, \cite[Definition 1.1]{gualtieri-annals}, the \emph{type} of $\mathcal{J}$ is given by the upper-semi-continuous function on $X$ defined by
$$ \mathrm{type}\left(\mathcal{J}\right) \;:=\; \frac{1}{2}\, \dim_\R \left( T^*X \cap \mathcal{J}T^*X \right) \;\in\; \left\{0, \ldots, n \right\} \;. $$
Points at which the type of the generalized-complex structure is locally constant are called \emph{regular points}.

A generalized Darboux theorem was proven by M. Gualtieri \cite[Theorem 4.35]{gualtieri-phd}, \cite[Theorem 3.6]{gualtieri-annals}. More precisely, for any regular point with type equal to $k$, there is an open neighbourhood endowed with a set of local coordinates such that the generalized-complex structure is a $B$-field transform of the standard generalized-complex structure of $\C^{k}\times\R^{2n-2k}$.

\subsubsection{Symplectic structures}
Symplectic structures can be interpreted as generalized-complex structures of type $0$ \cite[Example 4.10]{gualtieri-phd}.

\medskip

Let $X$ be a compact $2n$-dimensional manifold endowed with a symplectic structure $\omega \in \wedge^2 X$. The form $\omega\in\wedge^2X$ might be viewed as the isomorphism $\omega \colon TX \to T^*X$, which gives rise to the generalized-complex structure
$$ \mathcal{J}_\omega \;:=\;
\left(
\begin{array}{c|c}
 0 & -\omega^{-1} \\
\hline
 \omega & 0
\end{array}
\right) \;. $$

The $\Z$-graduation on forms is given by \cite[Theorem 2.2]{cavalcanti-jgp}
$$ U^{n-\bullet} \;=\; \exp{\left(\im\omega\right)}\, \left(\exp{\left(\frac{\Lambda}{2\im}\right)} \left(\wedge^\bullet X \otimes_\R \C\right)\right) \;, $$
where $\Lambda := -\iota_{\omega^{-1}}$.

By considering the isomorphism \cite[\S2.2]{cavalcanti-jgp}
$$ \varphi\colon \wedge X \otimes_\R \C \longrightarrow \wedge X \otimes_\R \C \;, \quad\text{where}\quad \varphi(\alpha) \;:=\; \exp{\left(\im\omega\right)}\, \left(\exp{\left(\frac{\Lambda}{2\im}\right)}\, \alpha\right) \;,$$
one has \cite[Corollary 1]{cavalcanti-jgp}
$$ \varphi\left(\wedge^\bullet X\otimes_\R\C\right) \simeq U^{n-\bullet} \;, $$
but also
$$ \varphi \, \de \;=\; \delbar \, \varphi  \quad\text{ and }\quad \varphi \, \de^{\Lambda} \;=\; -2\im\, \del\, \varphi \;, $$
where $\de^\Lambda := \left[\de,\,\Lambda\right]$, see \cite{koszul, brylinski}.

\subsubsection{Complex structures}
Complex structures can be interpreted as generalized-complex structures of type $n$ \cite[Example 4.11, Example 4.25]{gualtieri-phd}.

\medskip

Let $X$ be a compact $2n$-dimensional manifold endowed with a complex structure $J\in\End(TX)$. The complex structure gives rise to the generalized-complex structure
$$ \mathcal{J}_J
\;:=\;
\left(
\begin{array}{c|c}
 -J & 0 \\
\hline
 0 & J^*
\end{array}
\right)
\;\in\; \End\left(TX\oplus T^*X\right)
\;,$$
where $J^*\in\End(T^*X)$ denotes the dual endomorphism of $J\in\End(TX)$.

The $\Z$-graduation on forms is given by \cite[Example 4.25]{gualtieri-phd}
$$ U^\bullet_{\mathcal{J}_J} \;=\; \bigoplus_{p-q=\bullet}\wedge^{p,q}_JX \;. $$
Finally, note that $\del = \del_J$ and $\delbar = \delbar_J$ (see also \cite[Remark 4.26]{gualtieri-phd}).

\section{Generalized-complex subgroups of cohomologies}

Let $\mathcal{J}$ be a generalized-almost-complex structure on the manifold $X$.

Note that the differential $\de$ does not preserve the $\Z$-graduation $U^\bullet$. In fact, one can see that $GH_{dR}(X):=\frac{\ker\de}{\imm\de}$ is not $\Z$-graded. Hence, following \cite{li-zhang}, it is possible to force a $\Z$-graduation by studying the subgroups
$$ \left\{ GH_{\mathcal{J}}^{(k)}(X) \;:=\; \frac{\ker \de \cap\, U^k}{\imm \de} \right\}_{k\in\Z} \;. $$
They are denoted by $HH^k(X)$ and called {\em generalized cohomology} by G.~R. Cavalcanti in \cite[Definition at page 72]{cavalcanti-phd}.

In the integrable case, one can consider the natural map
$$ GH^\bullet_{BC}(X) \;:=\; \frac{\ker \del \cap \ker \delbar}{\imm \del\delbar}\, \longrightarrow\, GH_{dR}(X) $$
in such a way that, for any $k\in\Z$,
$$ GH^{(k)}_{\mathcal{J}}(X) \;=\; \imm \left( GH^k_{BC}(X) \to GH_{dR}(X) \right) \;. $$

\medskip

Note that
$$ \sum_{k\in\Z} GH_{\mathcal{J}}^{(k)}(X) \;\subseteq\; GH_{dR}(X) \;, $$
but in general, neither the sum is direct nor the inequality is an equality.

\begin{defi}
 Let $X$ be a compact manifold. A generalized-almost-complex structure $\mathcal{J}$ on $X$ is called
 \begin{itemize}
  \item {\em $\mathcal{C}^\infty$-pure} if
        $$ \bigoplus_{k\in\Z} GH_{\mathcal{J}}^{(k)}(X) \;\subseteq\; GH_{dR}(X) \;; $$
  \item {\em $\mathcal{C}^\infty$-full} if
        $$ \sum_{k\in\Z} GH_{\mathcal{J}}^{(k)}(X) \;=\; GH_{dR}(X) \;; $$
  \item {\em $\mathcal{C}^\infty$-pure-and-full} if it is both $\mathcal{C}^\infty$-pure and $\mathcal{C}^\infty$-full, that is,
        $$ \bigoplus_{k\in\Z} GH_{\mathcal{J}}^{(k)}(X) \;=\; GH_{dR}(X) \;. $$
 \end{itemize}
\end{defi}

\medskip

Analogously to \cite[Proposition 2.5]{li-zhang} and \cite[Theorem 2.1]{angella-tomassini-1}, we have the following proposition, assuring that $\mathcal{C}^\infty$-fullness is sufficient to have $\mathcal{C}^\infty$-pure-and-fullness (compare also with \cite[Proposition 4.1]{cavalcanti-phd}).

\begin{prop}\label{prop:Cf-Cpf}
 Let $X$ be a compact manifold endowed with a generalized-complex structure $\mathcal{J}$. If $\mathcal{J}$ is $\mathcal{C}^\infty$-full, then it is also $\mathcal{C}^\infty$-pure, and hence it is $\mathcal{C}^\infty$-pure-and-full.
\end{prop}

\begin{proof}
 Suppose that there exists 
\begin{equation}\label{a-in-h-k}
\mathfrak{a} \;\in\; GH^{(h)}_{\mathcal{J}}(X) \cap GH^{(k)}_{\mathcal{J}}(X)
\end{equation}
 with $h\neq k$. Let $\alpha^{h}\in U^h$ and $\alpha^{k}\in U^k$ such that $\mathfrak{a}=\left[\alpha^{h}\right]=\left[\alpha^{k}\right]$. By hypothesis, we have
 $$ GH_{dR}(X) \;=\; \sum_{\ell\in\Z} GH^{(\ell)}_{\mathcal{J}}(X) \;. $$
 
 Consider the Mukai pairing:
 $$ \left( \sspace , \ssspace \right) \colon \wedge^\bullet X\otimes\C \times \wedge^\bullet X\otimes\C \longrightarrow \mathcal{C}^\infty(X;\C) \;, \qquad \left( \varphi_1, \varphi_2 \right) \;:=\; \left( \sigma(\varphi_1)\wedge\varphi_2 \right)_{\text{top}} \;, $$
 where $\sigma$ acts on decomposable forms as $\sigma(e_1\wedge\cdots\wedge e_\ell):=e_\ell\wedge\cdots\wedge e_1$ and $(\sspace)_{\text{top}}$ denotes the top-dimensional component.
 By \cite[Proposition 2.2]{cavalcanti-phd}, one has that the previous pairing vanishes in $U^h \times U^k$ unless $h+k=0$, in which case it is non-degenerate.
 Furthermore, it can be seen that it induces a non-degenerate pairing in cohomology,
 $$ \left( \sspace , \ssspace \right) \colon GH_{dR}(X) \times GH_{dR}(X) \longrightarrow \C \;. $$
 
 However, 
 $$ \left( \mathfrak{a} , GH_{dR}(X) \right) \;=\; \sum_{\ell\in\Z} \left( \mathfrak{a} , GH^{(\ell)}_{\mathcal{J}}(X) \right) \;=\; 0 $$
 due to \eqref{a-in-h-k}. Thus, we can conclude that $\mathfrak{a}=0$.
\end{proof}

As a consequence, we have the following interpretation of $\mathcal{C}^\infty$-pure-and-fullness.

\begin{prop}\label{prop:Cpf-BC-surj-dR}
 A generalized-complex structure $\mathcal{J}$ on a compact manifold $X$ is $\mathcal{C}^\infty$-pure-and-full if and only if the natural map
 $$ \bigoplus_{k\in\Z} GH^k_{BC}(X) \longrightarrow GH_{dR}(X) $$
 induced by the identity is surjective.
\end{prop}

\medskip

We recall that a generalized-complex manifold is said to {\em satisfy the $\del\delbar$-Lemma}, \cite[Definition at page 70]{cavalcanti-phd}, if
$$ \imm\del \cap \ker\delbar \;=\; \imm\del\delbar \;=\; \imm\delbar \cap \ker\del \;. $$

Note that compact generalized-complex manifolds satisfying the $\del\delbar$-Lemma provide examples of $\mathcal{C}^\infty$-pure-and-full structures. In fact, by \cite[Theorem 4.2]{cavalcanti-phd}, they are $\mathcal{C}^\infty$-full, and by Proposition \ref{prop:Cf-Cpf}, they are also $\mathcal{C}^\infty$-pure.

\begin{thm}[{see \cite[Theorem 4.2]{cavalcanti-phd}}]
 Let $X$ be a compact manifold endowed with a generalized-complex structure $\mathcal{J}$. If it satisfies the $\del\delbar$-Lemma, then $\mathcal{J}$ is $\mathcal{C}^\infty$-pure-and-full.
\end{thm}

Finally, we prove that $B$-transforms preserve $\mathcal{C}^\infty$-pure-and-fullness.

\begin{prop}\label{B-trans}
 Let $B$ be a $\de$-closed $2$-form on a compact manifold $X$. The generalized-complex structure $\mathcal{J}$ on $X$ is $\mathcal{C}^\infty$-pure-and-full if and only if its $B$-transform $\mathcal{J}_B$ is.
\end{prop}

\begin{proof}
 Suppose that $\mathcal{J}$ is $\mathcal{C}^\infty$-pure-and-full. Let $\alpha \in GH_{dR}(X, \mathcal{J})$ and consider the $\de$-closed form $\exp(-B)\, \alpha$. By hypothesis, there exists a form $\gamma$ such that
 
 $$ \exp(-B)\, \alpha \;=\; \sum_{k\in\Z} \alpha^{k}_{\mathcal{J}} + \de\gamma, \quad \text{ where } \quad \alpha^{k}_{\mathcal{J}} \;\in\; U^k_{\mathcal{J}} \cap \ker \de \;. $$
 Therefore,
 $$ \alpha \;=\; \sum_{k\in\Z} \exp(B)\, \alpha^{k}_{\mathcal{J}} + \de \left(\exp(B)\,\gamma\right) \;. $$
 Since $\exp(B)\, \alpha^{k}_{\mathcal{J}} \in U^{k}_{\mathcal{J}_B} \cap \ker\de$ (see \cite[\S2]{cavalcanti-computations}), we can conclude that also $\mathcal{J}_B$ is $\mathcal{C}^\infty$-pure-and-full.

\smallskip
 The converse follows noting that $\mathcal{J}$ is the $(-B)$-transform of $\mathcal{J}_B$.
\end{proof}

As generalized-complex-geometry provides a common framework for both complex and symplectic geometry, one would expect to recover existing concepts of $\mathcal{C}^\infty$-pure-and-fullness for these two special cases. We devote to this aim the following lines.

\subsection{Symplectic subgroups of cohomologies}
Let $X$ be a compact manifold endowed with a symplectic structure $\omega$.
Denote $L\colon \wedge^\bullet X \ni \alpha \mapsto \alpha\wedge\omega \in \wedge^{\bullet+2}X$ and $\Lambda:=-\iota_{\omega^{-1}}$. Set also $P^\bullet := \ker \Lambda$.

A counterpart of the Bott-Chern cohomology in the symplectic case was introduced and studied by S.-T. Yau and L.-S. Tseng \cite{tseng-yau-1, tseng-yau-2, tseng-yau-3, tsai-tseng-yau}:
$$ SH^{\bullet}_{BC}(X) \;:=\; \frac{\ker\de\cap\ker\de^\Lambda}{\imm\de\de^\Lambda} \;, $$
where $\de^\Lambda:=\left[\de,\Lambda\right]$.

Inspired by Proposition \ref{prop:Cpf-BC-surj-dR}, we will say that, for every $k\in\Z$, the symplectic structure $\omega$ is {\em $\mathcal{C}^\infty$-pure-and-full at the $k$th stage in the sense of Brylinski} \cite{brylinski} if the natural map
$$ SH^{k}_{BC}(X) \longrightarrow H_{dR}^{k}(X;\R) $$
induced by the identity is surjective. When this holds at every stage, it means that every class in the de Rham cohomology admits a representative being both $\de$-closed and $\de^\Lambda$-closed: this property is known as {\em satisfying the Brylinski conjecture} \cite[Conjecture 2.2.7]{brylinski}.
By \cite[Corollary 2]{mathieu}, \cite[Theorem 0.1]{yan}, \cite[Proposition 1.4]{merkulov}, \cite{guillemin}, \cite[Theorem 5.4]{cavalcanti-phd}, it turns out that the following conditions are equivalent:
\begin{itemize}
 \item being $\mathcal{C}^\infty$-pure-and-full at every stage in the sense of Brylinski;
 \item satisfying the Brylinski conjecture;
 \item satisfying the Hard Lefschetz Condition;
 \item satisfying the $\de\de^\Lambda$-Lemma.
\end{itemize}
(Recall that a symplectic structure $\omega$ on a compact $2n$-dimensional manifold $X$ is said to satisfy the {\em Hard Lefschetz Condition} if, for any $k\in\Z$, the map $L^k \colon H^{n-k}_{dR}(X;\R) \to H^{n+k}_{dR}(X;\R)$ is bijective. Recall also that it is said to satisfy the {\em $\de\de^\Lambda$-Lemma} if every $\de^\Lambda$-closed, $\de$-exact form is $\de\de^\Lambda$-exact too; that is, if the natural map $SH^{\bullet}_{BC}(X) \to H^{\bullet}_{dR}(X;\R)$ is injective).

\medskip
Let us now show the explicit decomposition of $H^{\bullet}_{dR}(X;\R)$ in this case.

Consider the Lefschetz decomposition on the space of forms, $\wedge^\bullet X = \bigoplus_{2r+s=\bullet} L^rP^s$. Note that $\mathcal{C}^\infty$-pure-and-fullness at every stage in the sense of Brylinski means that the Lefschetz decomposition moves to cohomology. More precisely, one can define the following subgroups \cite{angella-tomassini-4}
$$ H^{(r,s)}_{\omega}(X) \;:=\; \left\{ \left[\alpha\right] \in H^\bullet_{dR}(X;\C) \st \alpha\in L^r P^s \right\} $$
and consider \cite{tseng-yau-1}
\begin{eqnarray*}
SH^{(r,s)}_{\omega}(X) &:=& L^r H^{(0,s)}_{\omega}(X) \\[5pt]
&=& \imm \left( L^r PH^s_{BC}(X) \to H^{2r+s}_{dR}(X;\R) \right) \\[5pt]
&=& \left\{ L^r\left[\beta^{(s)}\right] \in H^{2r+s}_{dR}(X;\R) \st \beta^{(s)} \in P^s \right\} \;\subseteq\; H^{2r+s}_{dR}(X;\R) \;,
\end{eqnarray*}
where
$$ PH^{\bullet}_{BC}(X) \;:=\; \frac{\ker\de\cap\ker\de^\Lambda\cap\, P^\bullet}{\imm\de\de^\Lambda} \;. $$
By \cite[Remark 2.3]{angella-tomassini-4}, it can be seen that $\omega$ is $\mathcal{C}^\infty$-pure-and-full in the sense of Brylinski if and only if
$$ H^{\bullet}_{dR}(X;\R) \;=\; \bigoplus_{2r+s=\bullet} SH^{(r,s)}_{\omega}(X) \;. $$

\medskip

Let us now compare the notions of $\mathcal{C}^\infty$-pure-and-fullness for a symplectic structure in the sense of Brylinski and for its induced generalized-complex structure.

\begin{prop}\label{prop:Cpf-sympl}
 Let $X$ be a compact manifold. Consider a symplectic structure $\omega$ on $X$, viewed as a generalized-almost-complex structure $\mathcal{J}$.
 Then $\mathcal{J}$ is $\mathcal{C}^\infty$-pure-and-full if and only if $\omega$ is $\mathcal{C}^\infty$-pure-and-full at every stage in the sense of Brylinski.
\end{prop}

\begin{proof}
 It suffices to observe that, in view of \cite[\S2]{cavalcanti-computations},
 $$ \varphi \;:=\; \exp(\im\omega)\,\exp\left(\frac{\Lambda}{2\im}\right) \colon \wedge^{n-\bullet}X \stackrel{\simeq}{\longrightarrow} U^{\bullet}_{\mathcal{J}} \;. $$
Furthermore,
 $$ \del_{\mathcal{J}} \;=\; -\frac{\im}{2}\, \varphi \circ \de^\Lambda \circ \varphi^{-1} \qquad \text{ and } \qquad \delbar_{\mathcal{J}} \;=\; \varphi \circ \de \circ \varphi^{-1} \;. $$
 
 Then we have the commutative diagram
 $$ \xymatrix{
  \Tot SH^{\bullet}_{BC}(X) \ar[r]^\simeq_{\varphi} \ar@{->>}[d] & \Tot GH^\bullet_{BC}(X) \ar@{->>}[d] \\
  \imm \left(\Tot SH^{\bullet}_{BC}(X) \to \Tot H^\bullet_{dR}(X;\R)\right) \ar[r]^{\varphi}_{\simeq} \ar@{^{(}->}[d] & \imm \left( \Tot GH^\bullet_{BC}(X) \to GH_{dR}(X) \right) \ar@{^{(}->}[d] \\
  \Tot H^\bullet_{dR}(X;\R) \ar@{=}[r] & GH_{dR}(X) \;.
 } $$
 
 This concludes the proof.
\end{proof}

\begin{rem}
 Concerning the notion of {\em $\mathcal{C}^\infty$-pure-and-fullness in the sense of \cite{angella-tomassini-4}}, that is, the property that
 $$ H^{\bullet}_{dR}(X;\R) \;=\; \bigoplus_{2r+s=\bullet} H^{(r,s)}_{\omega}(X) \;, $$
 we note that it is strictly weaker than the notion of $\mathcal{C}^\infty$-pure-and-fullness in the sense of Brylinski.
 In fact, by \cite[Theorem 2.6]{angella-tomassini-4}, every compact $4$-dimensional symplectic manifold is $\mathcal{C}^\infty$-pure-and-full in the sense of \cite{angella-tomassini-4}. On the other side, there are examples of such manifolds that do not satisfy the Hard Lefschetz Condition: hence, they are non-$\mathcal{C}^\infty$-pure-and-full in the sense of Brylinski. For example, consider non-tori nilmanifolds, \cite[Theorem A]{benson-gordon}, see also \cite[Theorem 9.2]{bock}.
 For a higher-dimensional example, see the results by M. Rinaldi in \cite{magda-thesis}.
\end{rem}

\subsection{Complex subgroups of cohomologies}

In the almost-complex case, T.-J. Li and W. Zhang introduced and studied the notion of $\mathcal{C}^\infty$-pure-and-fullness in \cite{li-zhang} (see \cite{draghici-li-zhang, angella-tomassini-1, angella-tomassini-2} and the references therein for further results).

More precisely, let $J$ be an almost-complex structure on the manifold $X$. For each $(p,q)\in\Z^2$, consider the subgroup
$$ H^{(p,q)}_{J}(X) \;:=\; \left\{ \left[\alpha\right] \in H^\bullet_{dR}(X;\C) \st \alpha\in\wedge^{p,q}X \right\} \;. $$
Given $k\in\Z$, the almost-complex structure $J$ is called {\em complex-$\mathcal{C}^\infty$-pure-and-full at the $k$th stage in the sense of Li and Zhang} \cite{li-zhang} if
$$ \bigoplus_{p+q=k} H^{(p,q)}_{J}(X) \;=\; H^k_{dR}(X;\C) \;. $$

Now we would like to compare the notions of $\mathcal{C}^\infty$-pure-and-fullness for a complex structure in the sense of Li and Zhang and for its induced generalized-complex structure.

\begin{lemma}
 Let $X$ be a compact manifold. Consider an almost-complex structure $J$ on $X$, viewed as a generalized-almost-complex structure $\mathcal{J}$. Then, for any $k\in\Z$, it holds
 $$ GH^{(k)}_{\mathcal{J}} \;=\; \bigoplus_{p-q=k} H^{(p,q)}_J(X) \;. $$
\end{lemma}

\begin{proof}
 Observe that in the almost-complex case the following equatily holds
 $$ U^k_{\mathcal{J}} \;=\; \bigoplus_{p-q=k} \wedge^{p,q}_JX \;. $$
 In general, the differential $\de$ of a generalized-complex structure can be decomposed as
 $$ \de \;=\; A_{\mathcal{J}} + \del_{\mathcal{J}} + \delbar_{\mathcal{J}} + \bar A_{\mathcal{J}} \colon U^k_{\mathcal{J}} \longrightarrow U^{k+2}_{\mathcal{J}} \oplus U^{k+1}_{\mathcal{J}} \oplus U^{k-1}_{\mathcal{J}} \oplus U^{k-2}_{\mathcal{J}} \,.$$
 However, as $\mathcal{J}$ is actually an almost-complex structure one has
 $$ \de \;=\; A_{J} + \del_{J} + \delbar_{J} + \bar A_{J} \colon \wedge^{p,q}_JX \longrightarrow \wedge^{p+2,q-1}_JX \oplus \wedge^{p+1,q}_JX \oplus \wedge^{p,q+1}_JX \oplus \wedge^{p-1,q+2}_JX $$
 and therefore,
 $$ A_{\mathcal{J}} \;=\; A_{J} \;, \qquad \del_{\mathcal{J}} \;=\; \del_{J} \;, \qquad \delbar_{\mathcal{J}} \;=\; \delbar_{J} \qquad \text{ and } \qquad \bar A_{\mathcal{J}} \;=\; \bar A_{J} \;. $$

 Take $[\alpha]\in GH^{(k)}_{\mathcal{J}}(X)$ with $\alpha = \sum_{p-q=k} \alpha^{(p,q)} \in U^k_{\mathcal{J}}$ and $\alpha^{(p,q)}\in\wedge^{p,q}_JX$. Then $A_{\mathcal{J}}\alpha=\del_{\mathcal{J}}\alpha=\delbar_{\mathcal{J}}\alpha=\bar A_{\mathcal{J}}\alpha=0$, but also $A_{J}\alpha=\del_{J}\alpha=\delbar_{J}\alpha=\bar A_{J}\alpha=0$. Thus, $\de \alpha^{(p,q)}=0$ for any $(p,q)$; that is, $\left[\alpha\right] = \sum_{p-q=k} \left[\alpha^{(p,q)}\right]$ where $\left[\alpha^{(p,q)}\right]\in H^{(p,q)}_{J}(X)$.
 The sum is obviously direct.
\end{proof}

As a consequence, we get the following.

\begin{prop}\label{prop:Cpf-cplx}
 Let $X$ be a compact manifold. Consider an almost-complex structure $J$ on $X$, viewed as a generalized-almost-complex structure $\mathcal{J}$.
 Then $\mathcal{J}$ is $\mathcal{C}^\infty$-pure-and-full if and only if $J$ is complex-$\mathcal{C}^\infty$-pure-and-full at every stage in the sense of Li and Zhang.
\end{prop}

\section{Generalized-complex structures on the differential nilmanifold underlying Iwasawa}\label{sec:iwasawa}

The {\em Iwasawa manifold} is the complex nilmanifold defined by
$$ \mathbb{I}_3 \;:=\; \left. \mathbb{H}(3;\Z[\im]) \middle\backslash \mathbb{H}(3;\C) \right. $$
where $\mathbb{H}(3;\C)$ is the $3$-dimensional \emph{Heisenberg group} over $\mathbb{C}$, that is,
$$
\mathbb{H}(3;\C) \;:=\; \left\{
\left(
\begin{array}{ccc}
 1 & z^1 & z^3 \\
 0 &  1  & z^2 \\
 0 &  0  &  1
\end{array}
\right) \in \mathrm{GL}(3;\mathbb{C}) \st z^1,\,z^2,\,z^3 \in\C \right\}
\;,
$$
and $\mathbb{H}(3;\Z[\im]) := \mathbb{H}(3;\C) \cap \mathrm{GL}(3;\Z[\im])$. It is worth to remark that it constitutes one of the simplest examples of non-K\"ahler complex manifold (see, e.g., \cite{fernandez-gray, nakamura}).

\medskip

In our case, we are interested in its underlying real nilmanifold that we will denote by $M=\Gamma\backslash G$. Following \cite{salamon}, let $\mathfrak{g}=(0,0,0,0,13+42,14+23)$ be the real nilpotent Lie algebra naturally associated to $G$ (i.e., the differentiable Lie group underlying $\mathbb{H}(3;\C)$). This notation means that $\mathfrak{g}^*$ admits a basis $\{e^k\}_{k=1}^6$ satisfying
$$ \left\{ \begin{array}{l}
            \de e^1 \;=\; \de e^2 \;=\; \de e^3 \;=\; \de e^4 \;=\; 0 \\[5pt]
            \de e^5 \;=\; e^{13} - e^{24} \\[5pt]
            \de e^6 \;=\; e^{14} + e^{23}
           \end{array} \right. \;,$$
where $e^{ij}:=e^i\wedge e^j$. These are known as the \emph{structure equations}.

By the Nomizu theorem, \cite[Theorem 1]{nomizu}, the de Rham cohomology of $M$ can be computed by means of the associated Lie algebra $\mathfrak{g}$ (i.e., using the previous structure equations). More precisely, given the Riemannian metric $g:=\sum_{j=1}^{6} e^j \odot e^j$, the harmonic representatives of the de Rham cohomology are the following:
\begin{eqnarray*}
 H^0_{dR}(M;\R) &=& \R \left\langle \left[1\right] \right\rangle \;, \\[5pt]
 H^1_{dR}(M;\R) &=& \R \left\langle \left[e^1\right],\, \left[e^2\right],\, \left[e^3\right],\, \left[e^4\right] \right\rangle \;, \\[5pt]
 H^2_{dR}(M;\R) &=& \R \left\langle \left[e^{12}\right],\, \left[e^{34}\right],\, \left[e^{15}-e^{26}\right],\, \left[e^{25}+e^{16}\right],\, \right. \\[5pt]
 && \left. \left[e^{35}-e^{46}\right],\, \left[e^{45}+e^{36}\right],\, \left[e^{13}+e^{24}\right],\, \left[e^{23}-e^{14}\right] \right\rangle \;, \\[5pt]
 H^3_{dR}(M;\R) &=& \R \left\langle \left[e^{125}\right],\, \left[e^{126}\right],\, \left[e^{345}\right],\, \left[e^{346}\right],\, \right. \\[5pt]
 && \left. \left[e^{135}-e^{245}-e^{236}-e^{146}\right],\, \left[e^{235}+e^{145}+e^{136}-e^{246}\right],\, \right. \\[5pt]
 && \left. \left[-e^{135}+e^{236}-e^{146}-e^{245}\right],\, \left[-e^{136}-e^{235}+e^{145}-e^{246}\right],\, \right. \\[5pt]
 && \left. \left[e^{135}+e^{245}+e^{236}-e^{146}\right],\, \left[-e^{235}+e^{145}+e^{136}+e^{246}\right]\right\rangle \;, \\[5pt]
 H^4_{dR}(M;\R) &=& \R \left\langle \left[e^{1256}\right],\, \left[e^{3456}\right],\, \left[e^{2346}-e^{1345}\right],\, \left[e^{1346}+e^{2345}\right],\, \right. \\[5pt]
 && \left. \left[e^{1246}-e^{1235}\right],\, \left[e^{1236}+e^{1245}\right],\, \left[e^{2456}+e^{1356}\right],\, \left[e^{1456}-e^{2356}\right] \right\rangle \;, \\[5pt]
 H^5_{dR}(M;\R) &=& \R \left\langle \left[e^{23456}\right],\, \left[e^{13456}\right],\, \left[e^{12456}\right],\, \left[e^{12356}\right] \right\rangle \;, \\[5pt]
 H^6_{dR}(M;\R) &=& \R \left\langle \left[e^{123456}\right] \right\rangle \;.
\end{eqnarray*}

Any linear complex structure $J$ defined on $\mathfrak{g}$ gives rise to a complex structure on $M$ that will be called \emph{left-invariant}. The Iwasawa manifold can be regarded as one of these structures, although there is an infinite family of them (see \cite{andrada-barberis-dotti, couv} for a complete classification up to isomorphism).

Let $\mathfrak{g}^{1,0}$ be the $i$-eigenspace of $J$ as an endomorphism on $\mathfrak{g}^*_\mathbb{C}:=(\mathfrak{g}\otimes_\R\mathbb{C})^*$. It is well-known that $J$ is a complex structure on $\mathfrak{g}$ if and only if  $\de(\mathfrak{g}^{1,0})\subset\wedge^{2,0}(\mathfrak{g}^*_\mathbb{C})\oplus\wedge^{1,1}(\mathfrak{g}^*_\mathbb{C})$.
There are two special types of complex structures that deserve our attention.
\begin{itemize}
\item $J$ is said to be \emph{holomorphically-parallelizable} if $\de(\mathfrak{g}^{1,0})\subset\wedge^{2,0}(\mathfrak{g}^*_\mathbb{C})$. In this case, $\mathfrak{g}$ can be endowed with a complex Lie algebra stucture and $M$ has a global basis of holomorphic vector fields.
\item $J$ is called \emph{Abelian} if $\de(\mathfrak{g}^{1,0})\subset\wedge^{1,1}(\mathfrak{g}^*_\mathbb{C})$. In this case, it turns out that $\mathfrak{g}^{1,0}$ is actually an Abelian complex Lie algebra.
\end{itemize}

From the general study accomplished in \cite{latorre-ugarte}, one can conclude that there are only two left-invariant complex structures defined on $M$ which are $\mathcal{C}^\infty$-pure-and-full at every stage in the sense of Li and Zhang. They are precisely the holomorphically-parallelizable stucture $J_0$ corresponding to the Iwasawa manifold (as proven in \cite{angella-tomassini-1}) and the Abelian stucture $J_1$ in \cite[Theorem 3.3]{andrada-barberis-dotti}. In the following lines, we give the explicit decomposition of the de Rham cohomology groups for each of these structures.

\medskip

With respect to the basis $\{e^k\}_{k=1}^6$, the complex structure $J_0$ can be defined as (see \cite{nakamura})
$$ J_0 e^1 \;=\; -e^2 \;, \qquad J_0 e^3 \;=\; -e^4 \;, \qquad J_0 e^5 \;=\; -e^6 \;. $$
Therefore, the forms
$$ \left\{ \begin{array}{rcl}
\varphi_{0}^1 &:=& e^1 + \im e^2 \;\stackrel{\text{loc}}{=}\; \de z^1 \;, \\[5pt]
\varphi_{0}^2 &:=& e^3 + \im e^4 \;\stackrel{\text{loc}}{=}\; \de z^2 \;, \\[5pt]
\varphi_{0}^3 &:=& e^5 + \im e^6 \;\stackrel{\text{loc}}{=}\; \de z^3 - z^1\, \de z^2
\end{array} \right.
$$
provide a left-invariant co-frame for the space of $(1,0)$-forms on $M$ with respect to $J_0$, with complex structure equations
$$
\left\{
\begin{array}{rcl}
 \de\varphi_{0}^1 &=& 0 \\[5pt]
 \de\varphi_{0}^2 &=& 0 \\[5pt]
 \de\varphi_{0}^3 &=& \varphi_{0}^1\wedge\varphi_{0}^2
\end{array}
\right. \;.
$$

As already stated, $J_0$ is complex-$\mathcal{C}^\infty$-pure-and-full at every stage in the sense of Li and Zhang \cite[Theorem 3.1]{angella-1}. In fact, it is possible to see that
\begin{eqnarray*}
 H^1_{dR}(M;\C) &=& \underbrace{\C \left\langle \left[\varphi_{0}^1\right],\, \left[\varphi_{0}^2\right] \right\rangle}_{H^{(1,0)}_{J_0}(M)} \,\ \oplus \,\ \underbrace{\C\left\langle \left[\varphi_{0}^{\bar1}\right],\, \left[\varphi_{0}^{\bar2}\right] \right\rangle}_{H^{(0,1)}_{J_0}(M)} \;, \\[5pt]
 H^2_{dR}(M;\C) &=& \underbrace{\C \left\langle \left[\varphi_{0}^{12}\right],\, \left[\varphi_{0}^{23}\right] \right\rangle}_{H^{(2,0)}_{J_0}(M)} \,\ \oplus \,\ \underbrace{\C\left\langle \left[\varphi_{0}^{\bar1\bar3}\right],\, \left[\varphi_{0}^{\bar2\bar3}\right] \right\rangle}_{H^{(0,2)}_{J_0}(M)} \\[5pt]
 & \oplus & \underbrace{\C\left\langle \left[\varphi_{0}^{1\bar1}\right],\, \left[\varphi_{0}^{1\bar2}\right],\, \left[\varphi_{0}^{2\bar1}\right],\, \left[\varphi_{0}^{2\bar2}\right] \right\rangle}_{H^{(1,1)}_{J_0}(M)} \;, \\[5pt]
 H^3_{dR}(M;\C) &=& \underbrace{\C\left\langle \left[\varphi_{0}^{123}\right] \right\rangle}_{H^{(3,0)}_{J_0}(M)} \,\ \oplus \,\ \underbrace{\C\left\langle \left[\varphi_{0}^{\bar1\bar2\bar3}\right] \right\rangle}_{H^{(0,3)}_{J_0}(M)} \\[5pt]
 & \oplus & \underbrace{\C\left\langle \left[\varphi_{0}^{13\bar1}\right],\, \left[\varphi_{0}^{13\bar2}\right],\, \left[\varphi_{0}^{23\bar1}\right],\, \left[\varphi_{0}^{23\bar2}\right] \right\rangle}_{H^{(2,1)}_{J_0}(M)} \\[5pt]
 & \oplus & \underbrace{\C\left\langle \left[\varphi_{0}^{1\bar1\bar3}\right],\, \left[\varphi_{0}^{1\bar2\bar3}\right],\, \left[\varphi_{0}^{2\bar1\bar3}\right],\, \left[\varphi_{0}^{2\bar2\bar3}\right]  \right\rangle}_{H^{(1,2)}_{J_0}(M)} \;, \\[5pt]
 H^4_{dR}(M;\C) &=& \underbrace{\C\left\langle \left[\varphi_{0}^{123\bar1}\right],\, \left[\varphi_{0}^{123\bar2}\right] \right\rangle}_{H^{(3,1)}_{J_0}(M)} \,\ \oplus \,\ \underbrace{\C\left\langle \left[\varphi_{0}^{1\bar1\bar2\bar3}\right],\, \left[\varphi_{0}^{2\bar1\bar2\bar3}\right] \right\rangle}_{H^{(1,3)}_{J_0}(M)} \\[5pt]
 & \oplus & \underbrace{\C\left\langle \left[\varphi_{0}^{13\bar1\bar3}\right],\, \left[\varphi_{0}^{13\bar2\bar2}\right],\, \left[\varphi_{0}^{23\bar1\bar3}\right],\, \left[\varphi_{0}^{23\bar2\bar3}\right] \right\rangle}_{H^{(2,2)}_{J_0}(M)} \;, \\[5pt]
 H^5_{dR}(M;\C) &=& \underbrace{\C\left\langle \left[\varphi_{0}^{123\bar1\bar3}\right],\, \left[\varphi_{0}^{123\bar2\bar3}\right] \right\rangle}_{H^{(3,2)}_{J_0}(M)} \,\ \oplus \,\ \underbrace{\C\left\langle \left[\varphi_{0}^{13\bar1\bar2\bar3}\right],\, \left[\varphi_{0}^{23\bar1\bar2\bar3}\right] \right\rangle}_{H^{(2,3)}_{J_0}(M)} \;, \\[5pt]
\end{eqnarray*}
where we have listed the harmonic representatives with respect to the Hermitian metric $g_0 := \sum_{j=1}^{3} \varphi_{0}^j \odot \bar\varphi_{0}^j$ and we have shortened, e.g., $\varphi_{0}^{1\bar1}:=\varphi_{0}^1\wedge \bar\varphi_{0}^1$.

\medskip

Following \cite[Theorem 3.3]{andrada-barberis-dotti}, one can define the Abelian complex structure $J_1$ with respect to the basis $\{e^k\}_{k=1}^6$ by
$$ J_1 e^1 \;=\; -e^3 \;, \qquad J_1 e^2 \;=\; -e^4 \;, \qquad J_1 e^5 \;=\; -e^6 \;. $$
Then, the forms
$$\tilde{\varphi}_{1}^1 \;:=\; e^1 + \im e^3 \;, \quad \tilde{\varphi}_{1}^2 \;:=\; e^2 + \im e^4 \;, \quad \tilde{\varphi}_{1}^3 \;:=\; e^5 + \im e^6 $$
provide a left-invariant co-frame for the space of $(1,0)$-forms on $M$ with respect to $J_1$, with complex structure equations
$$
\left\{
\begin{array}{rcl}
 \de\tilde{\varphi}_{1}^1 &=& 0 \\[5pt]
 \de\tilde{\varphi}_{1}^2 &=& 0 \\[5pt]
 \de\tilde{\varphi}_{1}^3 &=& 
       \frac{\im}{2}\,\tilde{\varphi}_1^{1}\wedge\tilde{\varphi}_1^{\bar{1}}-\frac{1}{2}\,\tilde{\varphi}_1^{1}\wedge\tilde{\varphi}_1^{\bar{2}}
       -\frac{1}{2}\,\tilde{\varphi}_1^{2}\wedge\tilde{\varphi}_1^{\bar{1}}-\frac{\im}{2}\,\tilde{\varphi}_1^{2}\wedge\tilde{\varphi}_1^{\bar{2}}
\end{array}
\right. \;.
$$
Applying the change of basis
$$\varphi_1^1=\tilde{\varphi}_1^1-\im\tilde{\varphi}_1^2, \qquad \varphi_1^2=\tilde{\varphi}_1^1+\im\tilde{\varphi}_1^2, \qquad 
  \varphi_1^3=2\im\tilde{\varphi}_1^3,$$
we obtain
$$
\left\{
\begin{array}{rcl}
 \de\varphi_{1}^1 &=& 0 \\[5pt]
 \de\varphi_{1}^2 &=& 0 \\[5pt]
 \de\varphi_{1}^3 &=& - \varphi_{1}^2\wedge\bar\varphi_{1}^1
\end{array}
\right. \;.
$$

Observe that these last equations yield to the following equivalent definition of the complex structure $J_1$:
$$ J_1 e^1 \;=\; -e^2 \;, \qquad J_1 e^3 \;=\; e^4 \;, \qquad J_1 e^5 \;=\; e^6 \;. $$

As previously said, $J_1$ is complex-$\mathcal{C}^\infty$-pure-and-full at every stage in the sense of Li and Zhang (see \cite[Proposition 4]{latorre-otal-ugarte-villacampa} as regards to the first stage, see also \cite{latorre-ugarte}). In fact, one has
\begin{eqnarray*}
 H^1_{dR}(M;\C) &=& \underbrace{\C\left\langle \left[\varphi_{1}^1\right],\, \left[\varphi_{1}^2\right] \right\rangle}_{H^{(1,0)}_{J_1}(M)} 
\,\ \oplus \,\ \underbrace{\C\left\langle \left[\varphi_{1}^{\bar1}\right],\, \left[\varphi_{1}^{\bar2}\right] \right\rangle}_{H^{(0,1)}_{J_1}(M)} \;, \\[5pt]
 H^2_{dR}(M;\C) &=& \underbrace{\C\left\langle \left[\varphi_{1}^{12}\right],\, \left[\varphi_{1}^{23}\right] \right\rangle}_{H^{(2,0)}_{J_1}(M)} \,\ \oplus \,\ \underbrace{\C\left\langle \left[\varphi_{1}^{\bar1\bar2}\right],\, \left[\varphi_{1}^{\bar2\bar3}\right] \right\rangle}_{H^{(0,2)}_{J_1}(M)} \\[5pt]
 & \oplus & \underbrace{\C\left\langle \left[\varphi_{1}^{1\bar1}\right],\, \left[\varphi_{1}^{2\bar2}\right],\, \left[\varphi_{1}^{1\bar3}\right],\, \left[\varphi_{1}^{3\bar1}\right] \right\rangle}_{H^{(1,1)}_{J_1}(M)} \;, \\[5pt]
 H^3_{dR}(M;\C) &=& \underbrace{\C\left\langle \left[\varphi_{1}^{123}\right] \right\rangle}_{H^{(3,0)}_{J_1}(M)} \,\ \oplus \,\ \underbrace{\C\left\langle \left[\varphi_{1}^{\bar1\bar2\bar3}\right] \right\rangle}_{H^{(0,3)}_{J_1}(M)} \\[5pt]
 & \oplus & \underbrace{\C\left\langle \left[\varphi_{1}^{12\bar3}\right],\, \left[\varphi_{1}^{13\bar1}\right],\, \left[\varphi_{1}^{23\bar1}\right],\, \left[\varphi_{1}^{23\bar2}\right] \right\rangle}_{H^{(2,1)}_{J_1}(M)} \\[5pt]
 & \oplus & \underbrace{\C\left\langle \left[\varphi_{1}^{1\bar1\bar3}\right],\, \left[\varphi_{1}^{1\bar2\bar3}\right],\, \left[\varphi_{1}^{2\bar2\bar3}\right],\, \left[\varphi_{1}^{3\bar1\bar2}\right] \right\rangle}_{H^{(1,2)}_{J_1}(M)} \;, \\[5pt]
 H^4_{dR}(M;\C) &=& \underbrace{\C\left\langle \left[\varphi_{1}^{123\bar1}\right],\, \left[\varphi_{1}^{123\bar3}\right] \right\rangle}_{H^{(3,1)}_{J_1}(M)} \,\ \oplus \,\ \underbrace{\C\left\langle \left[\varphi_{1}^{1\bar1\bar2\bar3}\right],\, \left[\varphi_{1}^{3\bar1\bar2\bar3}\right] \right\rangle}_{H^{(1,3)}_{J_1}(M)} \\[5pt]
 & \oplus & \underbrace{\C\left\langle \left[\varphi_{1}^{13\bar1\bar3}\right],\, \left[\varphi_{1}^{23\bar2\bar3}\right],\, \left[\varphi_{1}^{12\bar2\bar3}\right],\, \left[\varphi_{1}^{23\bar1\bar2}\right] \right\rangle}_{H^{(2,2)}_{J_1}(M)} \;,  
\\[5pt]
 H^5_{dR}(M;\C) &=& \underbrace{\C\left\langle \left[\varphi_{1}^{123\bar1\bar3}\right],\, \left[\varphi_{1}^{123\bar2\bar3}\right] \right\rangle}_{H^{(3,2)}_{J_1}(M)} \,\ \oplus \,\ \underbrace{\C\left\langle \left[\varphi_{1}^{13\bar1\bar2\bar3}\right],\, \left[\varphi_{1}^{23\bar1\bar2\bar3}\right] \right\rangle}_{H^{(2,3)}_{J_1}(M)} \;. \\[5pt]
\end{eqnarray*}
Again, we have listed the harmonic representatives with respect to the Hermitian metric $g_1 := \sum_{j=1}^{3} \varphi_{1}^j \odot \bar\varphi_{1}^j$ and we have shortened, e.g., $\varphi_{1}^{1\bar1}:=\varphi_{1}^1\wedge \bar\varphi_{1}^1$.

\medskip

As it is observed in \cite[Theorem 4.6, Theorem 1.3]{ketsetzis-salamon}, the space of left-invariant oriented complex structures on $M$ has the homotopy type of the disjoint union of a point and a $2$-sphere. However, G.~R. Cavalcanti and M. Gualtieri show in \cite{cavalcanti-gualtieri} that it is possible to connect these two disjoint components by means of a left-invariant generalized-complex structure $\rho$ of type $1$ on $M$. In fact, one can see that this structure precisely connects, up to $B$-transforms and $\beta$-transforms, our complex structures $J_0$ and $J_1$. As these are the only two left-invariant complex structures on $M$, up to equivalence, that are $\mathcal{C}^\infty$-pure-and-full as generalized-complex structures, it is natural to wonder what happens to $\rho$.

Next, we prove that $\rho$ is actually $\mathcal{C}^\infty$-pure-and-full and we provide another manner of joining $J_0$ and $J_1$ by means of left-invariant almost-complex structures on $M$. In contrast, we show that this new path is not $\mathcal{C}^\infty$-pure-and-full in the sense of Li and Zhang.

\subsection{Connecting generalized-complex structure by Cavalcanti and Gualtieri}\label{subsec:iwasawa-gencplx-path}

Consider the generalized-complex structure $\mathcal{J}$ given by G.~R. Cavalcanti and M. Gualtieri in \cite[\S5]{cavalcanti-gualtieri}. Observe that it is defined by the canonical bundle with generator
$$ \rho \;:=\; \exp \left(\im\,\left(-e^{36}-e^{45}\right)\right) \wedge \left( e^1+\im\,e^2\right) \;. $$

\medskip
Let us see that this structure is $\mathcal{C}^\infty$-pure-and-full by computing the subgroups $GH^{(k)}_{\mathcal{J}}(M) \;=\; \imm \left( GH^k_{BC}(M) \to GH_{dR}(M) \right)$. Observe that these calculations can be done at the Lie algebra level as a consequence of \cite{nomizu} and \cite{angella-calamai}.

\medskip
Consider the fibration
$$ \xymatrix{
 \left( \mathbb{T}^4,\, \omega \right) \ar@{^{(}->}[rr] && M \ar@{->>}[dd] \\
 (e^3,e^4,e^5,e^6) \ar@{|->}[r] & (e^1,\ldots,e^6) \ar@{|->}[d] & \\
 & (e^1,e^2) & \left( \mathbb{T}^2,\, J \right)
} $$
where
$$ \omega \;:=\; -e^{36}-e^{45} \qquad \text{ and } \qquad J \colon e_1 \longmapsto e_2 \;. $$
We note that, for every $k\in\Z$, one has
$$ U^k \;=\; \bigoplus_{r+s=k} U^r_J \otimes U^s_\omega \;. $$

Therefore, we compute
\begin{eqnarray*}
 U^1_J &=& \left\langle e^1+\im\,e^2 \right\rangle \;, \\[5pt]
 U^0_J &=& \left\langle 1,\; e^{12} \right\rangle \;, \\[5pt]
 U^{-1}_J &=& \left\langle e^1-\im\,e^2 \right\rangle \;, 
\end{eqnarray*}
and
\begin{eqnarray*}
 U^2_\omega    &=& \left\langle 1-\im\,e^{36}-\im\,e^{45}-e^{3456} \right\rangle \;, \\[5pt]
 U^1_\omega    &=& \left\langle e^{3}-\im\,e^{345},\; e^{4}+\im\,e^{346},\; e^{5}+\im\,e^{356},\; e^{6}-\im\,e^{456} \right\rangle \;, \\[5pt]
 U^0_\omega    &=& \left\langle e^{34},\; e^{35},\; e^{46},\; e^{56},\; e^{36}-e^{45},\; 1+e^{3456} \right\rangle \;, \\[5pt]
 U^{-1}_\omega &=& \left\langle e^{3}+\im\,e^{345},\; e^{4}-\im\,e^{346},\; e^{5}-\im\,e^{356},\; e^{6}+\im\,e^{456} \right\rangle \;, \\[5pt]
 U^{-2}_\omega &=& \left\langle 1+\im\,e^{36}+\im\,e^{45}-e^{3456} \right\rangle \;.
\end{eqnarray*}
From these computations, we get the claim. In fact, it can be seen that:
\begin{eqnarray*}
 GH^{(3)}_{\mathcal{J}}(M) &=& \C\left\langle \left[e^1+\im\,e^2-\im\,e^{136}-\im\,e^{145}+e^{236}+e^{245}-e^{13456}-\im\,e^{23456}\right] \right\rangle \;, \\[10pt]
 GH^{(2)}_{\mathcal{J}}(M) &=& \C\left\langle \left[e^{13}+\im\,e^{23}-\im\,e^{1345}+e^{2345}\right],\; \left[1-\im\,e^{36}-\im\,e^{45}-e^{3456}\right],\; \right. \\[5pt]
 && \left. \left[e^{15}+\im\,e^{16}+\im\,e^{25}-e^{26}+\im\,e^{1356}+e^{1456}-e^{2356}+\im\,e^{2456}\right],\; \right. \\[5pt]
 && \left. \left[e^{12}-\im\,e^{1236}-\im\,e^{1245}-e^{123456}\right] \right\rangle \;, \\[10pt]
 GH^{(1)}_{\mathcal{J}}(M) &=& \C\left\langle \left[e^1+\im\,e^2+e^{13456}+\im\,e^{23456}\right],\; \left[e^3-\im\,e^{345}\right],\; \left[e^{4}+\im\,e^{346}\right],\; \right. \\[5pt]
 && \left. \left[e^{125}+\im\,e^{12356}\right],\; \left[e^{126}-\im\,e^{12456}\right],\; \left[e^{135}-e^{146}+\im\,e^{235}-\im\,e^{246}\right],\; \right. \\[5pt]
 && \left. \left[e^{1}-\im\,e^{2}-\im\,e^{136}-\im\,e^{145}-e^{236}-e^{245}-e^{13456}+\im\,e^{23456}\right],\; \right. \\[5pt]
 && \left. \left[e^{136}-e^{145}-2\,\im\,e^{146}+\im\,e^{236}-\im\,e^{245}+2\,e^{246}\right] \right\rangle \\[10pt]
 GH^{(0)}_{\mathcal{J}}(M) &=& \C\left\langle \left[e^{13}+\im\,e^{23}+\im\,e^{1345}-e^{2345}\right],\; \left[e^{34}\right],\; \left[1+e^{3456}\right],\; \left[e^{1235}\right],\; \right. \\[5pt]
 && \left. \left[e^{1256}\right],\; \left[e^{12}+e^{123456}\right],\; \left[e^{13}-\im\,e^{23}-\im\,e^{1345}-e^{2345}\right],\; \left[e^{35}-e^{46}\right],\; \right. \\[5pt]
 && \left. \left[e^{15}+\im\,e^{16}+\im\,e^{25}-e^{26}-\im\,e^{1356}-e^{1456}+e^{2356}-\im\,e^{2456}\right],\; \right. \\[5pt]
 && \left. \left[e^{15}-\im\,e^{16}-\im\,e^{25}-e^{26}+\im\,e^{1356}-e^{1456}+e^{2356}+\im\,e^{2456}\right] \right\rangle \;, \\[10pt]
 GH^{(-1)}_{\mathcal{J}}(M) &=& \C\left\langle \left[e^{1}-\im\,e^{2}+e^{13456}-\im\,e^{23456}\right],\; \left[e^{3}+\im\,e^{345}\right],\; \left[e^{4}-\im\,e^{346}\right],\; \right. \\[5pt]
 && \left. \left[e^{125}-\im\,e^{12356}\right],\; \left[e^{126}+\im\,e^{12456}\right],\; \left[e^{135}-e^{146}-\im\,e^{235}+\im\,e^{246}\right],\; \right. \\[5pt]
 && \left. \left[e^{1}+\im\,e^{2}+\im\,e^{136}+\im\,e^{145}-e^{236}-e^{245}-e^{13456}-\im\,e^{23456}\right],\; \right. \\[5pt]
 && \left. \left[e^{136}-e^{145}+2\,\im\,e^{146}-\im\,e^{236}+\im\,e^{245}+2\,e^{246}\right] \right\rangle \;, \\[10pt]
 GH^{(-2)}_{\mathcal{J}}(M) &=& \C\left\langle \left[e^{13}-\im\,e^{23}+\im\,e^{1345}+e^{2345}\right],\; \left[1+\im\,e^{36}+\im\,e^{45}-e^{3456}\right],\; \right. \\[5pt]
 && \left. \left[e^{15}-\im\,e^{16}-\im\,e^{25}-e^{26}-\im\,e^{1356}+e^{1456}-e^{2356}-\im\,e^{2456}\right] ,\; \right. \\[5pt]
 && \left. \left[e^{12}+\im\,e^{1236}+\im\,e^{1245}-e^{123456}\right] \right\rangle \;, \\[5pt]
 GH^{(-3)}_{\mathcal{J}}(M) &=& \C\left\langle \left[e^{1}-\im\,e^{2}+\im\,e^{136}+\im\,e^{145}+e^{236}+e^{245}-e^{13456}+\im\,e^{23456}\right] \right\rangle \;.
\end{eqnarray*}

\subsection{Connecting almost-complex structures}

Let us start noting that the generalized-complex structure $\rho$ cannot be viewed as an almost-complex structure. For this aim, consider again the generalized-complex structure $\rho := \exp \left(\im\,\left(-e^{36}-e^{45}\right)\right) \wedge \left( e^1+\im\,e^2\right)$ by Cavalcanti and Gualtieri \cite[\S5]{cavalcanti-gualtieri}. For each $t\in[0,1]$, take the $B$-field
$$ B_t \;:=\; \exp(\pi\,\im\,t)\,\left( e^{35}-e^{46} \right) $$
and the $\beta$-field
$$ \beta_t \;:=\; -\frac{1}{4}\, \exp(\pi\,\im\,t)\, \left( e_3-\im\,\exp(\pi\,\im\,t)\,e_4 \right)\,\left( e_5-\im\,\exp(\pi\,\im\,t)\,e_6 \right) \;. $$
One has
$$ \exp(-\beta_t)\,\exp(-B_t)\,\rho \;=\; \left( e^1+\im\,e^2 \right)\wedge\left( e^3+\im\,\exp(\pi\,\im\,t)\,e^4 \right)\wedge\left( e^5+\im\,\exp(\pi\,\im\,t)\,e^6 \right) \;. $$
The endomorphism
$$ K_t \colon \left\{\begin{array}{rcl}
                    e^1 &\mapsto& - e^2 \\[5pt]
                    e^3 &\mapsto& - \exp(\pi\,\im\,t)\,e^4 \\[5pt]
                    e^5 &\mapsto& - \exp(\pi\,\im\,t)\,e^6
                   \end{array}\right. $$
is not an almost-complex structure for each $t\in[0,1]$.

\medskip

We now construct a curve of almost-complex structures on $M$ connecting the holomorphically-parallelizable structure $J_0$ and the Abelian complex structure $J_1$. Notice that, up to $\beta$-transforms and $B$-transform, it gives rise to a curve of generalized-almost-complex structures. We study $\mathcal{C}^\infty$-pure-and-fullness for the almost-complex structures.

For $t\in [0,1]$, consider the almost-complex structure

\setlength{\arraycolsep}{2pt}
\begin{equation}\label{eq:Jt}
J_t \;:=\;
 \left( \begin{array}{cc|cc|cc}
  & \ 1 \ & & & & \\
  \ -1 \ & & & & & \\
  \hline
  & & & \cos(\pi\,t) & & \sin(\pi\,t) \\
  & & -\cos(\pi\,t) & & \sin(\pi\,t) & \\
  \hline
  & & & -\sin(\pi\,t) & & \cos(\pi\,t) \\
  & & -\sin(\pi\,t) & & -\cos(\pi\,t) &
 \end{array} \right) \in \End(TM) \;.
\end{equation}

\medskip
Observe that the notation is coherent with the previous for $t=0$ and $t=1$. That is, $J_0$ coincides with the above holomorphically-parallelizable structure, and $J_1$, with the above Abelian complex structure.

 Consider
 \begin{eqnarray*}
 \rho_t &:=& \left( e^1 + \im\,e^2 \right) \wedge \left( e^3 + \im\,\left( \cos(\pi\,t)\,e^4 + \sin(\pi\,t)\,e^6 \right) \right) \\[5pt]
 && \wedge \left( e^5 - \im\,\left( \sin(\pi\,t)\,e^4 - \cos(\pi\,t)\,e^6 \right) \right) \;.
 \end{eqnarray*}
 In an equivalent manner, we can write it as:
 $$ \rho_t \;=\; \left( e^1 + \im\,e^2 \right) \wedge \left( -1 + \exp(\Xi) \wedge \exp(\im\,\omega_t) \right) $$
 where
 $$ \Xi \;:=\; e^{35} - e^{46} \quad \text{ and } \quad \omega_t \;:=\; \cos(\pi\,t)\,\left( e^{45} + e^{36} \right) - \sin(\pi\,t)\, \left( e^{34}+e^{56} \right) \;. $$
 Take $\beta:=e_{35}\in\wedge^2TM$. Then
 $$ \exp(\beta)\,\rho_t \;=\; \left(e^1+\im\,e^2\right) \wedge \exp(\Xi) \wedge \exp(\im\,\omega_t) \;. $$
 
 \smallskip\noindent
 Take also $B:=-\Xi \in \wedge^2M$. We obtain
 $$ \exp(B)\exp(\beta)\,\rho_t \;=\; \left(e^1+\im\,e^2\right) \wedge \exp(\im\,\omega_t) \;. $$
 
 \smallskip\noindent
 By computing $\left( e^1 + \im\,e^2 \right) \wedge \left( e^1 - \im\,e^2 \right) \wedge \omega_t^{2} = -4\im\,e^{123456} \neq 0$, we have that $\exp(B)\exp(\beta)\,\rho_t$ yields to a generalized-almost-complex structure of type $1$ on $M$, due to \cite[Theorem 4.8]{gualtieri-phd}.

 \medskip

We study complex-$\mathcal{C}^\infty$-pure-and-fullness in the sense of Li and Zhang for $J_t$.

\smallskip
Define the following basis of $(1,0)$-forms with respect to $J_t$:
$$ \left\{ \begin{array}{l}
            \varphi_t^1 \;=\; e^1 + \im\, e^2 \\[5pt]
            \varphi_t^2 \;=\; e^3 + \im\, \left( \cos(\pi\,t)\, e^4 + \sin(\pi\,t)\, e^6 \right) \\[5pt]
            \varphi_t^3 \;=\; e^5 - \im\, \left( \sin(\pi\,t)\, e^4 - \cos(\pi\,t)\, e^6 \right)
           \end{array} \right. \;.$$

The structure equations can be expressed as
$$ \left\{ \begin{array}{lll}
            \de \varphi_t^1 &=& 0 \\[5pt]
            \de \varphi_t^2 &=& \frac{1}{4}\, \sin(\pi\,t)\, \left( \left( 1+\cos(\pi\,t) \right)\, \varphi_t^{12} - \sin(\pi\,t)\, \varphi_t^{13} \right.
             \\[5pt] 
                             && \ \left. + \left( 1-\cos(\pi\,t) \right)\, \varphi_t^{1\bar{2}} + \sin(\pi\,t)\, \varphi_t^{1\bar{3}} 
                                       + \left( 1-\cos(\pi\,t) \right)\, \varphi_t^{2\bar{1}} \right. \\[5pt]
                             && \ \left. \sin(\pi\,t)\, \varphi_t^{3\bar{1}} - \left( 1+\cos(\pi\,t) \right)\, \varphi_t^{\bar{1}\bar{2}} 
                                       + \sin(\pi\,t)\, \varphi_t^{\bar{1}\bar{3}} \right) \\[5pt]
            \de \varphi_t^3 &=& \frac{1}{4}\, \left( \left( 1+\cos(\pi\,t) \right)^2\, \varphi_t^{12} - \sin(\pi\,t) \left( 1+\cos(\pi\,t) \right)\, \varphi_t^{13} \right. \\[5pt]
                             &&  \ \left. \left( 1-\cos^2(\pi\,t) \right)\,\varphi_t^{1\bar{2}} + 
                                              \sin(\pi\,t)\, \left( 1+\cos(\pi\,t) \right)\, \varphi_t^{1\bar{3}} \right. \\[5pt]
                             && \ \left. - \left( 1-\cos(\pi\,t) \right)^2\, \varphi_t^{2\bar{1}} - 
                                              \sin(\pi\,t)\, \left( 1-\cos(\pi\,t) \right)\, \varphi_t^{3\bar{1}} \right. \\[5pt]
                             && \ \left. \left( 1-\cos^2(\pi\,t) \right)\, \varphi_t^{\bar{1}\bar{2}} - 
                                              \sin(\pi\,t)\, \left( 1-\cos(\pi\,t) \right)\, \varphi_t^{\bar{1}\bar{3}} \right) 
           \end{array} \right. \;. $$

By direct computation, it is possible to see that, for $t\in \{0,\, 1\}$, one has
$$ H_{J_t}^{(1,0)}(M) \;=\; \C\left\langle \left[ \varphi_t^1 \right],\; \left[ \varphi_t^2 \right] \right\rangle \qquad\text{and}\qquad
  H_{J_t}^{(0,1)}(M) \;=\; \C\left\langle \left[ \bar\varphi_t^{1} \right],\; \left[ \bar\varphi_t^{2} \right] \right\rangle \;,$$
whereas, for $t\in(0,1)$,
$$ H_{J_t}^{(1,0)}(M) \;=\; \C\left\langle \left[ \varphi_t^1 \right] \right\rangle \qquad\text{and}\qquad
  H_{J_t}^{(0,1)}(M) \;=\; \C\left\langle \left[ \bar\varphi_t^{1} \right] \right\rangle \;. $$

Therefore, the almost-complex structures in the interior of the path \eqref{eq:Jt} joining the holomorphically-parallelizable and the Abelian complex structures on $M$ are not complex-$\mathcal{C}^\infty$-full at the first stage in the sense of Li and Zhang. Consequently, by Proposition \ref{prop:Cpf-cplx}, they are not $\mathcal{C}^\infty$-pure-and-full as generalized-almost-complex-structures.

\end{document}